\documentclass[11pt,a4paper,twoside,leqno]{amsart}
\usepackage{amsmath}
\usepackage{amsfonts}
\usepackage{amssymb}
\usepackage{amstext}
\usepackage{amsbsy}

\usepackage{graphicx}

\usepackage{epsf}
\theoremstyle{plain}
\begingroup

\newtheorem{thm}{Theorem}[section]
\newtheorem{prop}[thm]{Proposition}
\newtheorem{cor}[thm]{\textbf{\textit{Corollary}}}
\newtheorem{lem}[thm]{Lemma}

\endgroup
\newtheorem{rem}[thm]{\textbf{Remark}}

\newtheorem{defn}[thm]{Definition}

\newcommand{\ov}[1]{\overline{#1}}
\newcommand{\wt}[1]{\widetilde{#1}}

\newcommand{\hi}[1]{\mathbb{H}^#1}
\newcommand{\m}[1]{\mathbb{R}^#1}

\newcommand{\N}{\mathbb{N}}
\newcommand{\Ne}{\mathbb{N}^\ast}

\newcommand{\R}{\mathbb{R}}
\newcommand{\sd}{\mathbb{S}^2}

\newcommand{\hd}{\mathbb{H}^2}

\def\rmd{\mathop{\rm d\kern -1pt}\nolimits}

\begin{document}
\centerline{\large\bf }

\title[General curvature estimates for stable $H$-surfaces in 
3-manifolds ]{General curvature estimates for stable $H$-surfaces in 
3-manifolds and applications}

\author{\bf Harold Rosenberg \& Rabah Souam \&  Eric Toubiana }

\address{Instituto de Matem\'atica Pura e Aplicada (IMPA) \newline
 \indent Estrada Dona Castorina, 110  \newline
 \indent CEP 22460-320 \newline 
\indent Rio de Janeiro, RJ \newline
\indent Brasil}
\email{rosen@impa.br}

\address{ Institut de Math{\'e}matiques de Jussieu\newline
 \indent CNRS UMR 7586 - Universit{\'e} Paris Diderot - Paris 7\newline
 \indent G\'eom\'etrie et Dynamique \newline
 \indent Site Chevaleret \newline
\indent Case 7012       \newline  
 \indent         75205 - Paris Cedex 13, France}
\email{souam@math.jussieu.fr}
\email{toubiana@math.jussieu.fr}

\subjclass[2000]{53A10, 53C42}
\keywords{constant mean curvature, stability, curvature estimates, 
gradient estimates, bounded sectional curvature}


\begin{abstract}
 We obtain an estimate for the norm of the second fundamental 
form of stable $H$-surfaces
in  Riemannian 3-manifolds with bounded 
sectional curvature.
Our estimate depends on the distance to the boundary of the surface and
on the bound on the sectional curvature
but not on the 
manifold itself. We give some applications, in particular we obtain an 
interior
gradient estimate for $H$-sections in Killing submersions.

\end{abstract}


\date{\today}

\maketitle

\vskip8mm

\section{Introduction}\label{Sec.Intro}

 Let $(M,g)$ be a smooth Riemannian 3-manifold.
 Consider an immersed surface  $\Sigma \looparrowright M$ 
with trivial normal bundle. Call $A$ its second
fundamental form and $N$ a global unit normal on it.
We denote by $H$ the 
length of the mean curvature vector of $\Sigma$.
When $\Sigma$ has constant mean curvature $H$, we say that 
$\Sigma$ is an $H$-surface.

Recall that an $H$-surface $\Sigma$  is said {\em strongly stable}
if for any 
$u \in C_0^\infty (\Sigma)$ we have 
\begin{equation*}
\int_\Sigma \vert \nabla u\vert ^2 d\Sigma \geq 
\int_\Sigma (\vert A \vert^2 +\text {Ric}(N))u^2 d\Sigma,
\end{equation*}
where $d\Sigma$ and  $\nabla$ stand respectively  for the area 
element and the gradient on $\Sigma$, and Ric denotes the
Ricci curvature of $M$.
 Throughout this paper, we will use 
the term {\em stable} to mean strongly stable.

Curvature estimates for  stable 
$H$-surfaces in 3-manifolds are an important tool in the study of
$H$-surfaces, see for instance Meeks-Perez-Ros \cite{[MPR]}. Graphs are stable
and more
generally, surfaces transverse to an ambient Killing field are stable.
Curvature estimates for graphs were obtained in 1952 
by  Heinz \cite{[Heinz]}.
In 1983,  Schoen \cite{[Sch]}
obtained an estimate for the norm
of the second fundamental form of stable minimal surfaces in a 
3-manifold $M$ depending on the distance to the boundary on the surface and 
an upper bound on the curvature tensor of $M$ and its covariant derivative. In
particular, in $\R^3$, he showed the existence of a constant $C>0$ such that 
 for any stable and orientable minimal surface $\Sigma$ one has 
\begin{equation}\label{Eq,Schoen}
\vert A (p) \vert \leq \frac{C}{d_\Sigma (p, \partial \Sigma)},
\end{equation}
for any $p\in \Sigma$, where $d_\Sigma$ denotes the intrinsic distance on
$\Sigma$.

In 1999,  B\'erard and  Hauswirth \cite{[Be-Haus]} extended  Schoen's work
to stable  $H$-surfaces with trivial normal
bundle in space forms. In 2005,
using different methods based on ideas of  Colding and  Minicozzi 
\cite{[Cold-Min2]},
 Zhang \cite{[Zh]} extended Schoen's result to stable $H$-surfaces 
(with trivial normal bundle) in a general 3-manifold, again the estimate
depends on $H$, an upper bound on the  sectional curvature and on the 
covariant derivative of the curvature tensor of $M$.

In this paper we obtain an estimate for the norm of the second fundamental 
form of stable $H$-surfaces
in  Riemannian 3-manifolds assuming  only a bound on the 
sectional curvature.
Our estimate depends on the distance to the boundary of the surface and
only on the bound on the  sectional curvature of the ambient manifold. More
precisely, our
main result is the following:

\medskip 

\noindent {\bf Main Theorem.} {\em Let $(M,g)$ be a complete smooth Riemannian
3-manifold of  bounded sectional curvature 
$\vert K \vert \leq \Lambda < +\infty$.

 Then there exists a universal constant $C$ which depends neither on $M$ nor on 
$\Lambda$, satisfying the following:

For any immersed  stable $H$-surface
$\Sigma \looparrowright M$
 with trivial normal bundle,
and for any $p\in \Sigma$ 
we have
\begin{equation*}
\vert A(p) \vert \leq \frac{C}{\min \bigl\{ d(p,\partial \Sigma),
\frac{\pi}{2\sqrt{\Lambda}}\bigr\}}.
\end{equation*}
}

\bigskip

We also obtain a {\em local} version of our result, see
Theorem \ref{T.local}. Our method of proof is completely different 
from the above mentioned works 
and is based
on a blow-up argument.
For the readers convenience we sketch the idea of the proof.

\medskip

Assuming by
contradiction the result is not true, we have for each $n $ a  
stable $H_n$-surface $\Sigma_n$   in some 3-manifold
$M_n$ with bounded sectional curvature, $\vert K \vert \leq \Lambda$,
 admitting a point $p_n^*\in \Sigma_n$  satisfying
$\vert A_n (p_n^*) \vert \min \bigl\{ d(p,\partial \Sigma),
\frac{\pi}{2\sqrt{\Lambda}}\bigr\} > n$. 
In addition, there exists a real number $r_n>0$
such that 
the geodesic disk $D_n^*$ on $\Sigma_n$ centered at $p_n^*$ with  radius $r_n$
lies in the domain of some chart of $M_n$ which allows us to treat this disk as
a surface in a Euclidean ball of fixed radius 
 in $\R^3$ endowed with the pull-back metric, with
$p_n^*$ at the origin.
Furthermore in our construction we can make
 $r_n\vert A_n (p_n^*) \vert \to +\infty$. We then blow-up the latter metric 
multiplying by the factor $\vert A_n (p_n^*) \vert$. This sequence of new
metrics converges to the Euclidean metric on compact sets of $\R^3$. The
sequence $(D_n^*)$ gives rise to a new sequence of surfaces still denoted by 
$(D_n^*)$ with controlled geometry and  second fundamental form at the origin
of norm 1. We can then construct a complete surface $S$, passing 
through the origin, in the accumulation set of the sequence $(D_n^*)$, 
whose universal cover $\wt S$ is
a stable $H$-surface in the Euclidean space $\R^3$. 
It is well-known
that $S$ is then a plane which contradicts the fact that the second
fundamental form of $S$ has norm 1 at the origin.

However, this blow-up argument requires some care because of technical
difficulties we explain in the proof.

\medskip

It is worth noticing that an estimate of the type (\ref{Eq,Schoen}) cannot be
expected in general, even for minimal surfaces, as such an estimate would imply
that a complete  stable $H$-surface is totally geodesic. For example,
in $\hi3$  there are complete non-totally geodesic  stable $H$-surfaces,
see 
 Silveira \cite{[Sil]}. Also in $\hi2 \times \R$, for any $0\leq H\leq 1/2 $
there are non-constant entire vertical $H$-graphs, and therefore  stable, see
 Nelli and  Rosenberg,
\cite{[Nel-Ros1]} and \cite{[Nel-Ros2]}.

\bigskip

As a consequence of the Main Theorem we have the following:

\begin{cor}\label{C.complet}
 Let $(M,g)$ be a complete smooth Riemannian
3-manifold of bounded sectional curvature 
$\vert K \vert \leq \Lambda < +\infty$.

Then there exists a constant $C$, which depends neither on $M$ nor on $
\Lambda$, such that for any immersed 
complete stable $H$-surface
$\Sigma \looparrowright M$, with trivial normal bundle, we have:
\begin{itemize}
\item if $\Sigma$ is non compact, then $\vert A(p) \vert  \leq C \sqrt \Lambda$ for any $p\in \Sigma$,

\item if $\Sigma$ is compact, then  $\vert A(p) \vert \min\bigl\{\mathrm{diam}\, \Sigma,  \frac{\pi}{\sqrt{\Lambda}}\bigr\} \leq C$ for any $p\in \Sigma$, where $\mathrm{diam}\, \Sigma $  denotes the intrinsic diameter of $\Sigma.$
\end{itemize}
\end{cor}

\begin{proof}
For non compact surfaces, this follows immediately from the Main Theorem.

\smallskip

Suppose now that $\Sigma$ is compact. Set $d= \mathrm{diam}\, \Sigma $ and let $p\in \Sigma$, then 
 the geodesic disk
$D$ in $\Sigma$  of center $p$ and radius $d/2$ satisfies 
$d(p, \partial D)=d/2.$
Applying the Main Theorem to  $D$ we get  $\vert A(p) \vert \min\bigl\{\frac{d}{2},  \frac{\pi}{2\sqrt{\Lambda}}\bigr\} \leq C$. The result follows. 
\end{proof}

This corollary was enunciated by  Ros \cite{[Ros]}  for minimal surfaces (even
with non trivial normal bundle) with the additional hypotheses that the
derivatives of the curvature tensor are bounded and the injectivity radius is
bounded from below.

%

\bigskip

We are grateful to G. Besson, E. Hebey, J. Jost and M. Vaugon 
for useful information about harmonic coordinates.

\section{Proof of the Main Theorem }

In the blow-up argument we sketched before, 
the first idea which comes to mind is to use a chart given by the exponential
map, that is, to use geodesic normal coordinates. However in those coordinates
bounded geometry guarantees only a $C^0$-control of the metric. For example
consider the metrics $g_n, n \in \N$, on $\R^2$ given in polar coordinates by
\begin{equation*}
 g_n=dr^2 + G_n^2(r,\theta)d\theta^2,
\end{equation*}
with 
\begin{equation*}
 G_n(r,\theta)=r(1+ar^2 + e^{{-1}/{r^2}}\cos n\theta),
\end{equation*}
where $a>0$ is big enough. In Cartesian coordinates 
$x=r\cos \theta, y=r\sin \theta$, we have
\begin{align*}
 g_n &=\Big( \frac{x^2}{x^2+y^2} +\frac{y^2G_n^2}{(x^2+y^2)^2}\Big) dx^2
+2xy\Big( \frac{1}{x^2+y^2} - \frac{G_n^2}{(x^2+y^2)^2}\Big) dx dy \\
&\qquad \qquad \qquad \qquad
+\Big( \frac{y^2}{x^2+y^2} + \frac{x^2G_n^2}{(x^2+y^2)^2}\Big) dy^2 \\
&= g_{11}^ndx^2 + 2g_{12}^ndx dy + g_{22}^ndy^2.
\end{align*}
It is easily seen that $g_n$ is a complete and smooth metric on $\R^2$
and that $g_{ij}^n (0)=\delta_{ij}$  for any $n\in \N$, so that, $(x,y)$ 
are   geodesic normal coordinates at the origin for the metrics $g_n$.
The Gauss curvature of $g_n$ is given by 
$K_n=-\frac{1}{G_n}\frac{\partial^2 G_n }{\partial r^2}$. Thus, a computation
shows that
\begin{equation*}
 K_n=-\frac{6a +(\frac{4}{r^6}-\frac{2}{r^4})e^{{-1}/{r^2}}\cos n\theta }
{1+ar^2 + e^{{-1}/{r^2}}\cos n\theta}.
\end{equation*}
Therefore, for $a>0$ big enough, the Gaussian curvatures $K_n$ are negative
and 
uniformly bounded on $\R^2$ with respect to $n$. Consequently, the injectivity
radius 
of $(\R^2,g_n)$ is infinite for any $n\in \N$. We can show directly (or as a
consequence of Rauch comparaison theorem) that for any $r>0$ there exists  
$Q_0 >0$ depending on $r, a$ and not on $n$ such that on the geodesic balls
of $(\R^2,g_n)$ centered at the origin with radius $r$ we have 
\begin{equation*}
 Q_0^{-1} \delta_{ij}\leq g_{ij}^n \leq Q_0 \delta_{ij}, \quad \text{as
quadratic forms},
\end{equation*}
this gives a uniform $C^0$-control of the metrics $g_n$. Nevertheless we do
not have uniform $C^1$-control. Indeed, consider for example  the coefficient 
$g_{11}^n$, we have 
\begin{equation*}
 \frac{\partial g_{11}^n}{\partial x}=
 \frac{\partial }{\partial x} \big( \frac{x^2}{x^2+y^2}   \big)
+ \frac{\partial }{\partial x} \big( \frac{y^2}{(x^2+y^2)^2}\big)G_n^2
+ 2G_n \frac{y^2}{(x^2+y^2)^2} \frac{\partial G_n}{\partial x}.
\end{equation*}
Observe that at any fixed point $(x,y)\not=(0,0)$, there is only one term which
is not
uniformly bounded with respect to $n$: the one involving the partial derivative
of $G_n$ with respect to $x$. In coordinates $(x,y)$ we have 
\begin{equation*}
 G_n(x,y)=\sqrt{x^2 + y^2}+a (x^2 + y^2)^{3/2} +
 \sqrt{x^2 + y^2}e^{-1/(x^2 + y^2)} \frac{\text{Re} (x+iy)^n}{(x^2
+y^2)^{n/2}}.
\end{equation*}
Also, at any fixed point $(x,y)\not=(0,0)$, in the partial derivative 
$\partial G_n /\partial x$, the only term which is not uniformly bounded with
respect to $n$ is the one involving the derivative of 
$ {\text{Re} (x+iy)^n}/{(x^2 +y^2)^{n/2}} $. We have, 
using the relations $x=r\cos \theta, y=r\sin \theta$,
\begin{align*}
  \frac{\partial }{\partial x}\frac{\text{Re} (x+iy)^n}{(x^2 +y^2)^{n/2}} &=
n\frac{\text{Re} (x+iy)^{n-1}}{(x^2 +y^2)^{n/2}} -
nx\frac{\text{Re} (x+iy)^n}{(x^2 +y^2)^{1+n/2}}\\
&= \frac{n}{r} \cos (n-1)\theta -\frac{n}{r} \cos\theta \cos (n-1)\theta \\
&=\frac{n}{r} \sin\theta \sin n\theta.
\end{align*}
This shows that at any fixed point $(x,y)\not= (0,0)$ whose argument
$\theta$ is not rational with $\pi$ the partial derivative 
$(\partial g_{11}^n /\partial x)(x,y)$ is not uniformly bounded with respect to
$n$. Consequently uniform estimates of the Gaussian curvature do not imply 
a local $C^1$-control of the metric in geodesic normal coordinates.

\medskip

For this reason we will instead use {\em harmonic coordinates}. More precisely,
we will need the following result which can be deduced from
\cite[Theorem 6]{[Heb-Her]}.

\begin{thm}\label{T.cartes harmoniques}
Let $\alpha \in (0,1)$ and  $ \delta >0$.
Let $(M,g)$ be a smooth Riemannian 3-manifold, without boundary, of bounded
sectional curvature $\vert K \vert \leq \Lambda < +\infty$ and $\Omega$  an
open subset of $M$. Set 
\begin{equation*}
 \Omega (\delta)= \bigl\{ x\in M \mid d_g (x,\Omega)<\delta\bigr\}
\end{equation*}
where $d_g$ is the distance  asociated to $g$. Suppose there exists $i>0$ such
that for all $x\in  \Omega (\delta)$ we have $\text{inj}_{(M,g)}(x)\geq i$,
where $\text{inj}_{(M,g)}(x)$ is the injectivity radius at $x$.

Then there exist a constant $Q_0>1$ and a real
number $r_0 >0$, depending only on $i$, $\delta$,  $\Lambda$ and $\alpha$, and
not on $M$,
such that for any $x\in \Omega$, there exists a harmonic
coordinate chart 
$(U,\varphi, B(x,r_0))$, $U$ being an open subset of $\R^3$ containing the
origin
and  $B(x,r_0)$ the geodesic ball in $M$ centered at $x$ of radius $r_0$, 
with $\varphi (0)=x$,  and such that
the metric tensor $g$ is
$C^{1,\alpha}$-controlled. Namely the components $g_{ij}$, 
$i,j =1,2,3$ of $g$ satisfy:
\begin{align}
 Q_0^{-1} \delta_{ij} \leq g_{ij} \leq Q_0 \delta_{ij}, \quad \text{as
quadratic forms}, \label{quadratique}\\
\sum_{k=1}^3 \sup_{y \in U} \vert {\partial_k g_{ij} (y)}\vert 
+ \sum_{k=1}^3 \sup_{y \not=z }
 \frac{\vert  \partial_k g_{ij}(y)  - \partial_k g_{ij}(z) \vert }
{d_g(y,z)^\alpha} \leq Q_0. \label{norme}
\end{align}
\end{thm}

\bigskip

 We will also need the following result.

\begin{lem}\label{L.injectivity}
Let $(M,g)$ be a smooth complete Riemannian manifold 
of bounded sectional curvature, 
$\vert K \vert \leq \Lambda < +\infty$ -- thus for each  $x\in M$ 
the exponential map $\mathrm{exp}_x :
B_{\vec 0}(\frac{\pi}{\sqrt{\Lambda}}) \subset T_x M \rightarrow
B(x,\frac{\pi}{\sqrt{\Lambda}})$ is a local diffeomorphism.

 Then in the closed ball 
$\ov B_{\vec 0}(\frac{\pi}{4\sqrt{\Lambda}})$, the injectivity radius 
in  $B_{\vec 0}(\frac{\pi}{\sqrt{\Lambda}})$ endowed with
the pull-back metric $\mathrm{exp}_x ^{\ast}g$ is at least 
$\frac{\pi}{4\sqrt{\Lambda}}$.
\end{lem}

\begin{proof}
For any $p\in \ov  B_{\vec 0}(\frac{\pi}{4\sqrt{\Lambda}})$ we denote by $i(p)$
the injectivity radius of $p$  and by $C(p)$ its cut-locus in 
$(B_{\vec0}(\frac{\pi}{\sqrt{\Lambda}}),\text{exp}_x ^{\ast}g )$. 
We denote by $i_0$ the infimum of $i(p)$, 
$p\in \ov  B_{\vec 0}(\frac{\pi}{4\sqrt{\Lambda}})$. 

Assume by contradiction that $i_0 <\frac{\pi}{4\sqrt{\Lambda}}$. 
Let $(p_n) \in B_{\vec 0}(\frac{\pi}{4\sqrt{\Lambda}})$ be a sequence such
that 
$i(p_n) \to i_0$ and $i(p_n) < \frac{\pi}{4\sqrt{\Lambda}}$.
For each $n$ we take a point $q_n \in B_{\vec 0}(\frac{\pi}{2\sqrt{\Lambda}})$ 
such that $i(p_n)=d(p_n,q_n)$($=d(p_n,C(p_n)$), where $d$ is the distance in 
$(B_{\vec0}(\frac{\pi}{\sqrt{\Lambda}}),\text{exp}_x ^{\ast}g )$, we can assume
that $q_n \in B_{\vec 0} (\frac{3\pi}{8\sqrt{\Lambda}} + \frac{i_0}{2})$. 
Up to
choosing  subsequences, we can assume that $(p_n)$ converges to 
$p_\infty \in \ov  B_{\vec 0}(\frac{\pi}{4\sqrt{\Lambda}})$ and  
that $(q_n)$ converges to 
$q_\infty \in \ov  B_{\vec 0} (\frac{3\pi}{8\sqrt{\Lambda}} + \frac{i_0}{2})
\subset B_{\vec 0}(\frac{\pi}{2\sqrt{\Lambda}})$
and so $d(p_\infty, q_\infty)=i_0$. 

Note that $i_0 >0$. Indeed, there is a
neighborhood $W\subset B_{\vec 0}(\frac{\pi}{\sqrt{\Lambda}})$ of $p_\infty$ 
and there exists $\delta >0$ such that for any points $x, y$ in $W$ there is a
unique minimizing geodesic of length $<\delta$ joining $x$ to $y$, see do
Carmo \cite[Remark 3.8, p.72]{[DoC1]}. It follows 
easily that there exist a neighborhood $W^\prime \subset W$ of $p_\infty$ and 
$\delta^\prime >0$ such that for any $x\in W^\prime $ we have 
$i(x) \geq   \delta^\prime$, and so $i_0 \geq \delta^\prime$.

Since $d(p_n,q_n) < d(p_n, \partial B_{\vec 0}(\frac{\pi}{\sqrt{\Lambda}}))$, 
there exists for each $n$ 
a minimizing geodesic $\alpha_n$ in 
$B_{\vec 0}(\frac{\pi}{\sqrt{\Lambda}})$ joining $p_n$ to $q_n$.  
Actually, as $q_n \in C(p_n)$ and is not conjugate to $p_n$ 
(since $d(p_n,q_n)<\frac{\pi}{4\sqrt{\Lambda}}$)  there is another minimizing 
geodesic $\beta_n$ joining $p_n$ to $q_n$. As $d(p_n,q_n)=i(p_n)$, by a usual
argument, cf. for
instance Petersen \cite[Lemma 16, p.142]{[Pet]},  $\alpha_n$ and $\beta_n$ fit
smoothly 
together at $q_n$ to form a geodesic loop $\gamma_n$ (with possibly a corner at 
$p_n$). As the length of $\gamma_n$ is bounded from below by  $2i_0$, up to
taking a subsequence, $\gamma_n$ converges to a geodesic loop 
$\gamma_\infty$ at $p_\infty$ of length $2i_0=2d(p_\infty,q_\infty)$
 with midpoint $q_\infty$. 
Therefore $q_\infty \in C(p_\infty)$ and consequently 
$i(p_\infty)=i_0$.

 We now discuss two cases.

\medskip 

\noindent {\em Case 1.} Assume that $q_\infty \not\in
 \ov B_{\vec 0}(\frac{\pi}{4\sqrt{\Lambda}})$. 

Consider the geodesic sphere $S$ centered at $\vec 0$ of smallest radius
enclosing  the loop $\gamma_\infty$. Then $\gamma_\infty$ is tangent
to
$S$ at some smooth point $z$. We now see this is contradictory. For 
$x\in  B_{\vec 0}(\frac{\pi}{2\sqrt{\Lambda}})$, set $r(x)=d(x,\vec 0)$. 
Consider the
function $f(t)=r(\gamma_\infty (t))$, where $t$ is an arc length 
parameter of $\gamma_\infty$ with $\gamma_\infty  (0)=z$. 

We have 
$f^\prime (0)=0$ and 
$f^{\prime \prime}(t)=\nabla^2 
r (\gamma_\infty^\prime ,\gamma_\infty^\prime )+\langle 
\nabla r , \frac{D  }{dt }\gamma_\infty^\prime\rangle=
\nabla^2 
r (\gamma_\infty^\prime ,\gamma_\infty^\prime )$. By comparison theorems we
know that $\nabla^2 r$ is positive definite since $r(z) \leq
\frac{\pi}{4\sqrt{\Lambda}} +i_0 \leq \frac{\pi}{2\sqrt{\Lambda}}$, 
Petersen \cite[Chapter
6, Section 5, Theorem~27]{[Pet]}. Therefore $f$ has a local strict minimum
at $0$ contradicting the fact that $S$ encloses the loop $\gamma_\infty$.

\medskip 

\noindent {\em Case 2.} Assume that $q_\infty \in
 \ov B_{\vec 0}(\frac{\pi}{4\sqrt{\Lambda}})$.

By the symmetry of the cut-locus, $p_\infty \in C(q_\infty)$. Observe that we
 have by construction $i_0 \leq i(q_\infty)\leq  d(p_\infty,q_\infty)=i_0$.
Therefore $i(q_\infty)=d(p_\infty,q_\infty)$ and so $p_\infty$ is a closest
point to $q_\infty$ in $C(q_\infty)$. As before we deduce that $\gamma_\infty$
is also smooth at $p_\infty$. Taking the smallest geodesic sphere centered at 
$\vec 0$ enclosing the closed geodesic $\gamma_\infty$, we arrive to a
contradiciton as in case 1.
\end{proof}

\bigskip

We now start the proof of our Main Theorem.

\medskip

We first show that the conclusion of the theorem is true with a
constant $C(\Lambda)$ depending on $\Lambda$. Later we will show that this
constant may be chosen independant of $\Lambda$.

Assume by contradiction that  for any 
$n\in \Ne$ there are: a complete smooth Riemannian 3-manifold $(M_n,g_n)$
with bounded sectional curvature, $\vert K \vert \leq \Lambda$, 
a  stable $H_n$-surface 
$\Sigma_n  \looparrowright     M_n $ 
and a point $p_n \in \Sigma_n$ satisfying 
\begin{equation}\label{Eq.inegalite1}
\vert A_n (p_n) \vert  \min ( d_{\Sigma_n}(p_n, \partial\Sigma_n),
\frac{\pi}{2\sqrt{\Lambda}}) \geq n,
\end{equation}
$ A_n$ being the second fundamental form of $\Sigma_n$.

We denote by $D_n \subset \Sigma_n$ the open geodesic
disk of radius 
$\min ( d_{\Sigma_n}(p_n, \partial\Sigma_n), \frac{\pi}{2\sqrt{\Lambda}} )$
centered at $p_n$.

For any $n\in \Ne$ let us denote by ${\vec 0_n}$ the origin in $T_{p_n}M_n$. 
Call  $\wt D_n$ the connected component, containing  $\vec 0_n$, of
$\text{exp}_{p_n}^{-1}
(D_n) \cap B_{\vec 0_n}(\frac{\pi}{8\sqrt{\Lambda}})\subset T_{p_n}M_n $. 
We endow the ball $B_{\vec 0_n}(\frac{\pi}{\sqrt{\Lambda}})$ with the
pull-back metric 
$\text{exp}_{p_n}^\ast g_n$. By abuse of notation, we still denote by $g_n$
this metric and note that its sectional curvature satisfies 
$\vert K \vert \leq \Lambda$.

Observe that $\wt D_n$ is a stable $H_n$-surface with trivial normal bundle in 
$B_{\vec 0_n}(\frac{\pi}{\sqrt{\Lambda}})$. Indeed, by a result of
Fischer-Colbrie and Schoen \cite{[FC-Sc]}, there exists a
positive Jacobi function $u$ on $D_n$, therefore 
$u\circ \text{exp}_{p_n}$ is a positive Jacobi function on 
$\wt D_n$, this implies stability of $\wt D_n$ (cf. \cite{[FC-Sc]}). By abuse
of notation we still denote by $A_n$ the second fundamental form of 
$\wt D_n$ in $(B_{\vec 0_n}(\frac{\pi}{\sqrt{\Lambda}}), g_n)$.

Note that 
\begin{equation}\label{Eq.inegalite1a} 
 \vert A_n (\vec 0_n) \vert \min \bigl\{d_{\wt D_n} (\vec 0_n, \partial \wt
D_n), \frac{\pi}{2\sqrt{\Lambda}}\bigr\} \geq n.
\end{equation}
Indeed, we have $\vert A_n (\vec 0_n) \vert =\vert A_n (p_n) \vert$
and one can check that for any $x\in \partial \wt D_n$ we have 
$d_{\wt D_n} (\vec 0_n, x) \geq d_{ D_n} (p_n, \partial D_n)$ (consider the two
cases $x\in \partial  B_{\vec 0_n}(\frac{\pi}{4\sqrt{\Lambda}})$ and \linebreak
$x\not\in \partial  B_{\vec 0_n}(\frac{\pi}{4\sqrt{\Lambda}})$). Then we
conclude by (\ref{Eq.inegalite1}).

Let $\alpha \in (0,1)$ be a fixed number. Consider the constants $Q_0>1$ and
$r_0$ given in Theorem \ref{T.cartes harmoniques} applied to 
$(M,g)= (B_{\vec 0_n}(\frac{\pi}{\sqrt{\Lambda}}),g_n)$, 
$\Omega=B_{\vec 0_n}(\frac{\pi}{8\sqrt{\Lambda}})$,  
$\delta=\frac{\pi}{8\sqrt{\Lambda}}$ and $i=\frac{\pi}{4\sqrt{\Lambda}}$ (cf. 
Lemma \ref{L.injectivity}).
We can assume that 
$r_0\leq \frac{\pi}{4\sqrt{\Lambda}}$.
Let $p_n^\ast$ be a point in the closure of $\wt D_n$ where the function 
$f: \wt D_n \rightarrow \R$ defined by 
$f(x)= \vert A_n(x)\vert \min \bigl\{  d_{\wt D_n }(x,\partial \wt D_n),
r_0\bigr\}$ reaches its
maximum. Note that $p_n^\ast$ is an interior point of $\wt D_n$ since 
$f\equiv 0$ on $\partial \wt D_n$.

One can check that $f(\vec 0_n) \geq r_0 \frac{\sqrt{\Lambda}}{\pi} n$. As 
$f(\vec 0_n) \leq f(p_n^\ast) \leq \vert A_n(p_n^\ast)\vert r_0$ we get 
\begin{equation}\label{Eq.ineg1}
\vert A_n(p_n^\ast)\vert \geq \frac{\sqrt{\Lambda}}{\pi} n.
\end{equation}

Put $\rho_n = \min \bigl\{d_{\wt D_n}(p_n^\ast,\partial \wt D_n), r_0\bigr\}$. 
 Consider the geodesic disk 
$D_n^\ast \subset \wt D_n$ of radius $\rho_n/2$
centered at 
$p_n^\ast$.

Consider for any $n \in \Ne$ the harmonic chart of $M_n$,
$(U_n, \varphi_n, B(p_n^\ast, r_0))$, given by 
Theorem \ref{T.cartes harmoniques},
where $U_n\subset \R^3$ is an open neighborhood of the origin
and $\varphi_n (0)=p_n^\ast$.
By the property (\ref{quadratique}) the Euclidean ball 
 of radius $r_0/Q_0 $ centered at the 
origin, $B_e (0,r_0/Q_0 )$, 
 is contained in $U_n$ for any $n$.  

Call $\lambda_n = \vert A_n(p_n^\ast)\vert$, by 
(\ref{Eq.ineg1}), 
 $\lambda_n \to +\infty$. 
Since $\lambda_n = \vert A_n(p_n^\ast)\vert \geq H_n$, the
sequence $H_n /\lambda_n$ is bounded and so, up to taking a subsequence, we can
assume that 
\begin{equation}\label{Eq.limiteH}
\frac{H_n}{\lambda_n} \to H^\ast < \infty.
\end{equation}
We can also assume that
 $(\lambda_n)$ is nondecreasing. Let 
$F_n$ be the homothety of $\m3$ of ratio $1/\lambda_n$. Put
$V_n =F_n^{-1} (U_n)$ and observe that 
$B_e (0,\lambda_n r_0/Q_0 ) \subset V_n \subset B_e (0,Q_0 \lambda_n r_0)$ for
any
$n$. Therefore $\cup_n V_n =\m3$ and,
up to taking a subsequence, we can assume that the sequence
$(V_n) $ is monotone for the inclusion.

We call $(x_1,x_2,x_3)$ the cartesian coordinates of $x\in U_n$ and 
$(y_1,y_2,y_3)$ those of $y=F_n^{-1}(x)\in V_n$.
Now we use a blow-up argument,
we endow $V_n$ with the metric $h_n=\lambda_n^2 F_n^\ast ( \varphi_n^\ast
g_n)$. 
In the coordinates $(x_1,x_2,x_3)$ the metric $ \varphi_n^\ast g_n $ reads as 
\begin{equation*}
 (\varphi_n^\ast g_n)(x) =\sum _{i,j} g_{ij}^n (x) dx_i dx_j .
\end{equation*}

For any $C^{1,\alpha}$-function $w$ on an open set $\Omega \subset \R^3$
we set:
\begin{equation*}
 \Vert w\Vert_{C^{1,\alpha}(\Omega)} = 
\Vert w \Vert_\infty  + \sum_{k=1}^3 \sup_{y\in \Omega} 
\Big\vert {\partial_k w}(y) \Big\vert +
\sum_{k=1}^3 \sup_{y\not= z} \frac{\big\vert  {\partial_k w}(y)
- {\partial_k w}(z)   \big\vert}{\vert y-z \vert^\alpha} .
\end{equation*}

Observe that, because of property (\ref{quadratique}) of Theorem 
\ref{T.cartes harmoniques}, up to passing to a subsequence, we can assume 
that the sequence of inner products $(g_{ij}^n (0))$ converges to some
inner product on $\R^3$. Up to a linear change of coordinates we can assume
that this limit is the Euclidean inner product $(\delta_{ij})$.

From properties (\ref{quadratique}) and (\ref{norme})  of Theorem 
\ref{T.cartes harmoniques}, there exists a constant
$Q>1$ such that :
\begin{equation}\label{Eq.estimees}
 \Vert g_{ij}^n  \Vert_{C^{1,\alpha}(U_n)} \leq Q .
\end{equation}

For each $n$ and each $y\in V_n$ we
have 
\begin{equation*}
 h_n (y)=\sum _{i,j} g_{ij}^n (\frac{y}{\lambda_n}) dy_i dy_j .
\end{equation*}
Thus the components $h_{ij}^n$ of $h_n$ in the coordinates 
$y=(y_1,y_2,y_3)$
are given by $h_{ij}^n(y)=g_{ij}^n (\frac{y}{\lambda_n}) $ 
for any $y\in V_n$. 
Observe that $\frac{\partial h_{ij}^n}{\partial y_k}(y)=
\frac{1}{\lambda_n}\frac{\partial g_{ij}^n}{\partial
x_k}(\frac{y}{\lambda_n})$, so we get on $V_n$:
\begin{equation}
\Big\vert \frac{\partial h_{ij}^n}{\partial y_k} \Big\vert \leq 
\frac{Q}{\lambda_n}, \quad k=1,2,3.
\end{equation}

 Therefore we have for any $y\in V_n$ 
\begin{equation}
 \vert h_{ij}^n (y)-\delta_{ij} \vert\leq
\sqrt{3}\frac{Q}{\lambda_n} \vert y \vert  +
\vert h_{ij}^n (0)-\delta_{ij} \vert .
\end{equation}

\bigskip 

Thus on any Euclidean ball $B_R\subset V_n$ of radius $R>0$ centered at the
origin, we
get 
\begin{equation}
 \Vert h_{ij}^n -\delta_{ij} \Vert_{C^{1,\alpha}(B_R)} \leq
\sqrt{3}\frac{Q}{\lambda_n}  R + 3\frac{Q}{\lambda_n}
+\frac{Q^{1+\alpha}}{\lambda_n^{1+\alpha}} +\vert h_{ij}^n (0)-\delta_{ij}
\vert .
\end{equation}
It follows that the sequence $(V_n, h_n)$ converges uniformly on compact
subsets of $\m3$ for
the $C^{1,\alpha}$-Euclidean topology to $(\m3, g_{euc})$, where $g_{euc}$
stands for the Euclidean metric. 

Composing with the diffeomorphism $ (\varphi_n \circ F_n)^{-1}$ we view 
$D_n^\ast$ as immersed in $V_n$. 
Note that the image $ (\varphi_n \circ F_n)^{-1}(p_n^\ast)=0\in \m3$ for any
$n$ and that $D_n^\ast$ is a geodesic disk of radius  $\lambda_n \rho_n/2$
centered at $p_n^\ast$ for the metric induced by $h_n$.

Note that $\lambda_n \rho_n=f(p_n^\ast) \geq f(\vec 0_n) \geq r_0 
\frac{\sqrt{\Lambda}}{\pi} n$, thus
\begin{equation}\label{Eq.complete}
 \lambda_n \rho_n \to +\infty .
\end{equation}

For $x \in D_n^\ast$ we have 
$d_{\wt D_n}(p_n^\ast, \partial  \wt D_n)\leq 2 
d_{\wt D_n}(x, \partial \wt D_n)$. Since $f(x)\leq f(p_n^\ast)$ 
we deduce that the second fundamental form $ A_n^\ast$ of
$D_n^\ast$ in $(V_n,h_n)$ satisfies 
\begin{equation}\label{Eq.inegalite2}
\vert A_n^\ast (x)\vert =\frac{\vert A_n (x)\vert}{\lambda_n} \leq 
\frac{\min \bigl\{d_{\wt D_n}(p_n^\ast,\partial \wt D_n), r_0\bigr\}}
{\min \bigl\{d_{\wt D_n}(x,\partial \wt D_n), r_0\bigr\}} \leq 2.
\end{equation}
It is important to observe that
\begin{equation}\label{Eq.nonflat}
 \vert  A_n^\ast (p_n^\ast) \vert =1
\end{equation}
for any $n\in \Ne$.

Let us call $\text{II}_n$ the second fundamental form of $D_n^\ast$ 
in $(\m3,g_{euc})$.

Given $m\in \Ne$ denote by $\Delta_{n,m}$ the connected component of 
$D_n^\ast \cap B_m$ passing through 0 (that is, containing $p_n^\ast$).
As the sequence of metrics $h_n$ converges 
on compact sets of $\m3$ to $g_{euc}$ for the $C^{1,\alpha}$-Euclidean topology,
we deduce from Proposition \ref{P.geometrie} in the appendix that 
there
exists $n_m\in \Ne$ such that for any $n\geq n_m$ and any $x\in \Delta_{n,m}$,
we have 
\begin{equation*}
\Big\vert \vert \text{II}_n (x)\vert - \vert A_n^\ast (x)\vert \Big\vert \leq
\frac{1}{m} .
\end{equation*}
Furthermore we have 
$ d_{ \Delta_{n,m} }(p_n^\ast, \partial \Delta_{n,m})\geq m$,
if $n_m$ is big enough,
since $\lambda_n \rho_n\to +\infty$. Observe that if $m^\prime >m$
 then $\Delta_{n,m}\subset  \Delta_{n,m^\prime}$. Moreover we can assume that
the sequence $(n_m)$ is increasing. We set $\Delta_m=\Delta_{n_m,m}$. Thus, 
we have constructed 
in this way a sequence of connected
surfaces $\Delta_n\looparrowright \m3$ passing through 0 (that is, containing
$p_n^\ast$) with the following property:  
for any $k\in \Ne$ there
exists $n_k$ such that for any $n\geq n_k$ and any $x\in \Delta_n\cap B_k$ we
have 
\begin{equation}\label{Eq.inegalite3} 
\Big\vert \vert \text{II}_n (x)\vert - \vert A_n^\ast (x)\vert \Big\vert \leq
\frac{1}{k} \quad \text{and} \quad
d_{\Delta_n} (p_n^\ast,\partial \Delta_n)\geq k.
\end{equation}
In particular we get for $n\geq n_k$ 
\begin{equation}
 \Big\vert \vert \text{II}_n(p_n^\ast)\vert -1 \Big\vert \leq \frac{1}{k}.
\end{equation}
Furthermore,
for any $k$, any $n\geq n_k$ and any 
$x\in \Delta_n \cap B_k$ we get from  
(\ref{Eq.inegalite2}) and
(\ref{Eq.inegalite3}):
\begin{equation}\label{Eq.estimee}
 \vert \text{II}_n (x) \vert \leq 5.
\end{equation}

We will use the following well known result which can be deduced from
Colding-Minicozzi \cite{[C-M]} or Perez-Ros \cite{[P-R]}.

\begin{prop}\label{P.graph} 
Let $\Sigma \looparrowright \m3$ be an immersed surface whose 
second fundamental form $A$ satisfies
$\vert A \vert <\frac{1}{4\delta}$ for  some constant $\delta >0$.
Then 
for any $x\in \Sigma$ with $d_\Sigma (x,\partial \Sigma)>4\delta$ there
is 
a neighborhood of $x$ in $\Sigma$ which is a graph of a function $u$ over the 
Euclidean disk of radius $\sqrt{2}\delta$ centered at $x$ in the tangent plane
of
$\Sigma$ at $x$. Moreover $u$ satisfies 
\begin{equation}\label{Eq.estimates}
\vert u \vert <2 \delta,\quad 
 \vert \nabla u \vert <1 \quad \text{and} \quad 
\vert \nabla^2 u \vert < \frac{1}{\delta} .
\end{equation}

\end{prop}

\medskip

We deduce from the estimate (\ref{Eq.estimee}) and Proposition \ref{P.graph}
that for each $n $ big enough, a part of the surface $\Delta_n$  is the
graph of a function $u_n$ over a Euclidean disk of radius $\sqrt{2}\delta$,
with $\delta =1/20$,
centered at the origin in the tangent plane of $\Delta_n$ at the origin.
Furthermore the functions $u_n$ satisfy the uniform estimates
(\ref{Eq.estimates}). 
Up to passing to a subsequence, still denoted $\Delta_n$,
and up to a rotation in $\R^3$, we can assume that the tangent planes $T_0
\Delta_n$ converge to the horizontal plane $P$ through the origin.
Consequently, 
for $n$ big enough, a part of $\Delta_n$  is the graph of a
function still denoted $u_n$,  over the Euclidean disk $D_\delta$
of radius
$\delta$ centered at the origin in $P$. By continuity,
note that the new functions $u_n$ satisfy the following uniform estimates for
$n$ big enough:
\begin{equation}\label{Eq.new estimates}
\vert u_n \vert < 3 \delta,\quad 
 \vert \nabla u_n \vert <2 \quad \text{and} \quad 
\vert \nabla^2 u_n \vert < \frac{1}{\delta} \qquad (\text{with}\ 
\delta=\frac{1}{20}).
\end{equation}
Thus we have obtained uniform $C^2$-estimates for the functions $u_n$ on
$D_\delta$. To go further we
need $C^{2,\alpha}$-estimates, $0<\alpha<1$. This is the content of the next
lemma.

\begin{lem}
For any $\delta^\prime \in (0,\delta)$ there exits a constant 
$C$ which does not depend on $n$ such that for n big enough
we have 
\begin{equation}
  \Vert u_n \Vert_{C^{2,\alpha}(D_{\delta^\prime})} < C,
\end{equation}
where $D_{\delta^\prime}$ denotes the disk in $P$ of radius 
$\delta^\prime$ centered at the origin.
\end{lem}

\begin{proof}
Since the mean curvature of $\Delta_n$ for the metric $h_n$ is 
$H_n/\lambda_n$, we infer, see for instance Colding-Minicozzi 
\cite[p. 99--100]{[C-M]}, 
that the function $u_n$ is solution of 
an elliptic PDE of the form:
\begin{equation}\label{Eq.mean curvature}
L(u):=a^{ij}(u,\nabla u,h_n)u_{ij} 
+b^i (u,\nabla u, h_n,{\partial_m h_n})u_i =
2\frac{H_n}{\lambda _n}W(\nabla u, h_n) +
c(u,\nabla u, h_n,
{\partial_m h_n}),
\end{equation}
where by $h_n$ we mean the components $(h_n)_{\alpha \beta}$ of the metric
$h_n$ and by ${\partial_m h_n}$ we mean the partial derivatives
$\frac{\partial (h_n)_{\alpha \beta}}{\partial y_m}$, 
$1\leq \alpha, \beta, m \leq 3$. Moreover the functions $a^{ij}, b^i$, 
$W$ and $c$ depend in a smooth way on their arguments.

Recall that the sequence $(u_n)$ has uniform $C^2$-estimates on 
 $D_{\delta}$ and so has uniform $C^{1,\alpha}$-estimates. Moreover the
sequence of metrics $(h_n)$ converges to the Euclidean metric $g_{euc}$ for the 
$C^{1,\alpha}$-topology on compact sets.

Recall also that 
$H_n /\lambda_n$ is bounded, see (\ref{Eq.limiteH}).
 Thus the functions
$a^{ij}(u,\nabla u,h_n)$, $b^i (u,\nabla u, h_n,
{\partial_m h_n})$ and 
$2({H_n}/{\lambda _n})W(\nabla u, h_n) +
c(u,\nabla u, h_n,
{\partial_m h_n})$ have uniform $C^{0,\alpha}$-estimates
on  $D_{\delta}$, that is,  $C^{0,\alpha}$-estimates independant on $n$.

 By Schauder estimates, see Gilbarg-Trudinger \cite[Chapter 6]{[Gi-Tr]} or
 Petersen \cite[Chapter 10]{[Pet]},
for any 
$\delta^\prime \in (0,\delta)$, there exists a constant $C_0>0$ which does not
depend on $n$ such that:
\begin{equation}
   \Vert u_n \Vert_{C^{2,\alpha}(D_{\delta^\prime})}
< C_0 \Big( \Vert L(u_n) \Vert _{C^\alpha (D_{\delta})} 
+\Vert u_n \Vert_{C^\alpha (D_{\delta})}  \Big).
\end{equation}
 Consequently, there exists a constant $C>0$ which does not depend on $n$
such that: 
\begin{equation}
  \Vert u_n \Vert_{C^{2,\alpha}(D_{\delta^\prime})} < C,
\end{equation}
which concludes the proof of the lemma.
\end{proof}

Fix some $\delta^\prime \in (0,\delta)$. As the sequence $(u_n)$ is 
$C^{2,\alpha}$-bounded on the disk $D_{ \delta^\prime }$, by 
Arzela-Ascoli's theorem there is some subsequence, still denoted $(u_n)$, which
converges to a $C^2$-function $u$ in the $C^2$-topology. 
Since the mean curvature, $H_n/\lambda_n$, of the graph of $u_n$ for the metric 
$h_n$ tends to $H^\ast$, see (\ref{Eq.limiteH}), thanks to the Proposition
\ref{P.geometrie} we know that the
mean curvature of
the graph of $u_n$ for the Euclidean metric also tends to $H^\ast$. We deduce 
that the graph of $u$, denoted by $S$, is an $H^\ast$-surface for the
Euclidean metric. 

Note that this graph 
contains the origin and that its second fundamental form $A$ verifies: 
$\vert A(0) \vert =1$ and   $\vert A \vert <5$  on 
$D_{ \delta^\prime }$ 
(by (\ref{Eq.inegalite2}), (\ref{Eq.nonflat}) and Proposition
\ref{P.geometrie}).

\medskip 

Let $x_0\in D_{ \delta^\prime }$ at Euclidean distance $\delta^\prime/2$ 
from 0 and fix once for all some  
$\delta^{\prime \prime}\in (\delta^\prime, \delta)$.
Note that for $n$ big enough, a part of the surface $\Delta_n$ is the graph
of a function $v_n$ over 
the Euclidean disk $D_{\delta ^{\prime \prime}}(x_0)$ centered at $(x_0,
u(x_0))\in S$ in the tangent plane 
$T_{(x_0, u(x_0))} S$. Using the same arguments as before, we can extract a
subsequence of $(v_n)$ whose graphs converge to a $H^\ast$-graph over
the disk 
$D_{\delta ^{\prime }}(x_0)$. Observe that this new graph is not contained in
$S$ because $\vert \nabla u \vert \leq 1$. We have in this way obtained a new
$H^\ast$-surface extending $S$.
We will still 
denote this extended  $H^\ast$-surface by $S$.

As $d_{\Delta_n}(0,\partial \Delta_n)\to +\infty$, using a standard 
diagonal process, we obtain a complete $H^\ast$-surface $S$ immersed in
$\R^3$, 
passing through the origin and whose second fundamental form $A$ satisfies:
\begin{equation}\label{Eq.contradiction}
 \vert A(0) \vert =1 \quad \text{and} \quad  \vert A \vert \leq 5.
\end{equation}

Consider the universal covering $\wt S$ of $S$, observe that 
$\wt S$ is naturally immersed in $\R^3$ and its second fundamental form 
is bounded. Consequently, 
there exists an $\varepsilon >0$ such that 
the map $\Pi: W:= \wt S \times (-\varepsilon,+\varepsilon) \rightarrow \R^3$
given by
$\Pi (p,t)=p+ t \wt N(p)$ is an immersion, where $\wt N$ is the Gauss map 
of $\wt S$. Therefore we can endow $W$ with a flat metric which makes $\Pi$ a
local
isometry. Note that $\wt S$ is a complete $H^\ast$-surface in $W$.

\medskip

\noindent {\bf Claim:} $\wt S$ is stable in $W$.

For any geodesic disk in $S$ of radius $\delta^\prime$, for $n$ big
enough, a piece of $\Delta_n$ is a graph over this geodesic disk of $S$. By
construction of $S$, using this fact,  for any 
compact and simply connected domain $U$ in $ S$,  a piece of $\Delta_n$
will be a graph over $U$ for $n$ big enough.
Using  a continuation
argument, for any  
compact and simply connected domain $\wt U$ in $ \wt S$,
 we can lift a part $G_n$ of $\Delta_n$, $n$ big enough, to get a surface 
$\wt G_n$ in $W$ 
which will be a graph over $\wt U$.
Moreover the graphs $\wt G_n$  converge to $\wt U$ in the $C^2$-topology.

Observe that each graph $\wt G_n$ is a stable $H$-surface (with
$H=H_n/\lambda_n$) in (a part of) $W$
endowed with the metric $\Pi^\ast (h_n)$.
Indeed, since $ G_n$ is two-sided and
 stable for the metric $h_n$, 
by a result of Fischer-Colbrie and Schoen \cite{[FC-Sc]}, there exists a
positive Jacobi function $v_n$ on $G_n$. The function 
$\wt v_n =v_n \circ \Pi$ is thus a positive Jacobi function on 
$\wt G_n$  for the metric $\Pi^\ast (h_n)$. Again by \cite{[FC-Sc]}
this implies that $\wt G_n$ is stable. 

For $n$ big enough, $\wt G_n$ is the graph over $\wt U$ of a
function which tends to 0 in the $C^2$-norm as $n$ goes to $+\infty$. In this
way we may see the stability operator of $\wt G_n$: 
$J_n= \Delta_n + \vert \wt{B}_n^* \vert^2 + \text{Ric}_n(\nu)$ as an 
operator on $\wt U$. Here $\wt{B}_n^* $ stands for the second fundamental form 
of $\wt G_n$, 
$\Delta_n$ denotes the Laplacian on $\wt G_n$ and $ \text{Ric}_n(\nu)$ denotes
the Ricci curvature in the direction of 
the unit normal field $\nu$ to $\wt G_n$, all of them with respect to
the metric $\Pi^* (h_n)$.

 As the metrics $\Pi^\ast (h_n)$ 
converge to the flat metric $\Pi^\ast (g_{euc})$ in the 
$C^{1,\alpha}$-topology and 
$\vert \text{Ric}_n(\nu)\vert  \leq 2\Lambda /\lambda_n^2$,
it follows that the domain $\wt U$ is stable for the
flat metric. This implies that $\wt S$ is stable in $W$ for the flat metric as
claimed.

\medskip 

As the immersion $\Pi$ is a local isometry we infer that the 
complete immersion $\wt S \rightarrow S\subset \R^3$  is stable. 
Thanks to results of do Carmo-Peng \cite{[DoC-Pe]},
Fischer-Colbrie and Schoen \cite{[FC-Sc]}, Pogorelov \cite{[Pog]}, 
Lopez-Ros \cite{[Lopez-Ros]} and Silveira \cite{[Sil]},  
we know
that $S$ is a plane. This contradicts the fact that $\vert A(0)\vert =1$, 
see (\ref{Eq.contradiction}).

It remains to check that $C(\Lambda)$ can be chosen independant of 
$\Lambda$. 

First we can assume that $C(\Lambda)$ is the infimum among the constants
satisfying the conclusion of the theorem. Let $\Sigma$ be any 
stable $H$-surface in $(M,g)$. Observe that $\Sigma$  is a stable 
$\frac{H}{\tau}$-surface in $(M,\tau^2 g)$ for any $\tau >0$ and that the
quantity $\vert A(p) \vert {\min \bigl\{ d(p,\partial \Sigma),
\frac{\pi}{2\sqrt{\Lambda}}\bigr\}}$ is scale invariant. It follows that 
$C(\Lambda)=C(\Lambda/\tau^2)$ for any $\tau >0$ and so $C(\Lambda)$ does not
depend on $\Lambda$, which concludes the proof of the Main Theorem. 
\qed 

\bigskip

The same proof gives the
following {\em local} result.

\begin{thm}\label{T.local}
 Let $(M,g)$ be a smooth Riemannian 3-manifold 
$($not necessarily complete$)$, 
with
bounded sectional curvature 
$\vert K \vert \leq \Lambda < +\infty$. Let
$\Omega$ be an open subset of $M$ such that there exists $\delta>0$
for which $ \Omega (\delta)$ is relatively compact in $M$ 
$($cf. Theorem $\ref{T.cartes harmoniques}$ for the notation$)$.

 Then 
there exists a constant $C= C(\delta^2 \Lambda)>0$ 
depending only on the product  $\delta^2\Lambda$ and neither on $M$
nor on $\Omega$, 
satisfying the following:

For any immersed  stable $H$-surface
$\Sigma \looparrowright \Omega$,
with trivial normal bundle, 
and for any $p\in \Sigma$ 
we have
\begin{equation*}
 \vert A(p) \vert < \frac{C}{\min \bigl\{ d(p,\partial \Sigma),
\frac{\pi}{2\sqrt{\Lambda}}, \delta\bigr\}}.
\end{equation*}
\end{thm}

\section{Applications}

In this section we consider Riemannian 3-manifolds $(M^3,g)$ which
fiber over a Riemannian surface $(M^2,h)$. We assume that the fibration 
$\Pi : (M^3,g) \rightarrow (M^2,h)$ is a Riemannian submersion with the
following properties.

\begin{enumerate}
\item \label{Complete} Each fiber is a complete geodesic  
of infinite length. 
\item \label{Killing} The fibers of the fibration are the integral curves of a
unit Killing vector field $ \xi$ on $M^3$.
\end{enumerate}
A fibration satisfying (\ref{Complete}) and (\ref{Killing}) will be called a
{\em Killing submersion}. It can be shown that such a fibration is 
(topologically) trivial.
 Indeed, there always exists a global section $s : M^2 \rightarrow M^3$
(see Steenrod \cite[Theorem 12.2]{[Steenrod]}).  Considering 
the flow $\varphi_t$ of 
$\xi$, a trivialization of the fibration is given by the diffeomorphism:
$(p,t)\in M^2 \times \R \rightarrow \varphi_t (s(p)) \in M^3$.

Notice that there are many such 3-manifolds, including $\R^3$, 
$\wt{\textrm{PSL}(2,\R)}$ (which fibers over the hyperbolic plane $\hd$), the
Heisenberg space $\text{Nil}_3$ (which fibers over the Euclidean plane $\R^2$)
and the metric product spaces $M^2 \times \R$ for any Riemannian surface 
$(M^2,h)$.

\medskip

\begin{defn}
 Let $\Pi : (M^3,g) \rightarrow (M^2,h)$ be a Killing submersion.
\begin{enumerate}
 \item Let $\Omega \subset M^2$ be a domain. An  $H$-{\em section} over
$\Omega$ is an $H$-surface in $M^3$
which is 
the image of a section $s : \ov \Omega \rightarrow M^3$, with $s$ of class 
$C^2$ on $\Omega$ and $C^0$ on
 $\ov\Omega$.

\item Let $\gamma \subset M^2$ be a smooth curve with geodesic curvature
$2H$. Observe that the surface $\Pi^{-1}(\gamma) \subset M^3$ has mean
curvature $H$. We call  such a surface a {\em vertical} $H$-{\em cylinder}.
\end{enumerate}

 We may also consider
$H$-sections without boundary.
\end{defn}

\begin{rem}\label{R.graphes}
 Let $s :\ov \Omega \rightarrow \Sigma \subset \Omega$ 
be a $H$-section  where 
$\Omega \subset M^2$ is a  relatively compact domain. 
We make the following observations.
\begin{enumerate}
\item \label{item.transverse} For any interior point $p\in \Omega$ the
$H$-surface 
$\Sigma$ is transversal at  $s(p)$ to the fiber. Indeed, assume by
contradiction that $\Sigma$ is tangent to the fiber at the interior  point
$s(p)$. Let $S\subset M^3$ be the vertical $H$-cylinder tangent to $\Sigma$ at 
$s(p)$ with the same mean curvature vector.
Then, in a neighborhood of $s(p)$, the intersection $\Sigma \cap S$ is
composed of $n\geq 2$ smooth curves passing through $s(p)$. But the union of
those curves cannot be a graph, contradicting the assumption that 
$\Sigma$ is a graph over $\ov\Omega$.

\item \label{item.estimes} Let $s_1 :\ov \Omega \rightarrow \Sigma_1 \subset
\Omega$ 
be another $H$-section over $\ov \Omega$ such that the mean curvature vector 
field $\vec H_1$ points in the same direction as the mean curvature vector 
field $\vec H$ of $\Sigma$, that is, the scalar products 
$g(\vec H, \xi)$ and $g(\vec H_1, \xi)$ have the same sign. Assume that there
exists a constant $C>0$ such that $\vert s(p)-s_1(p) \vert <C$ at any 
$p\in \partial \Omega$ where $\vert s(p)-s_1(p)\vert$ denotes the distance on
the fiber over $p$ between the points $s(p)$ and $s_1(p)$. As $\xi$ is a Killing
field, the translated copies of $\Sigma$ along the fibers are also $H$-surfaces
whose  mean curvature vector  has the same orientation as that 
 of $\Sigma$. 
Then applying the maximum principle to $\Sigma_1$ and a translation of 
$\Sigma$ by $\xi$, 
we deduce  that we
have also $\vert s(p)-s_1(p) \vert <C$ at any $p \in \ov \Omega$. 

This gives for any such $H$-section $\Sigma_1$ 
height estimates,  relative to $\Sigma$, 
depending on the vertical distance between the boundaries of $\Sigma$ and
$\Sigma_1$.
\end{enumerate}
\end{rem}

Our first application is as follows.

\begin{thm}\label{T.courbe}
 Let $\Pi : (M^3,g) \rightarrow (M^2,h)$ be a Killing submersion 
and let  $s: \Omega \rightarrow \Sigma \subset M^3$ be an $H$-section
over a domain $\Omega \subset M^2$. Let $U_0 $ be a neighborhood of an arc 
$\gamma \subset \partial \Omega$ and $s_0 : U_0 \rightarrow M^3$ a section.

Assume that 
 for any sequence $(p_n)$ of $\Omega$ which converges to a point
$p\in \gamma$, the height of $s(p_n)$ goes to $+\infty$, that is 
$s(p_n)-s_0(p_n) \to +\infty$.

 Then, $\gamma$ is  a smooth curve with geodesic curvature $2H$. 
If $H>0$ then
$\gamma$ is convex with respect to $\Omega$ if, and only if,  
the mean curvature vector $\vec{H}$ of $\Sigma$ points up,
that is, if $g(\vec H, \xi)>0$ along $\Sigma$. Moreover,
$\Sigma$ converges to the vertical $H$-cylinder  
$\Pi^{-1}(\gamma)$
 with respect to the $C^k$-topology for any $k\in \N$; this convergence will be
made precise in the proof.
\end{thm}

\begin{proof}

Let $p\in \gamma$ and $(p_n)$ a sequence in $\Omega$ converging to $p$. 
Let  $\ov{B(p,\rho)}$ be a compact geodesic disc of $M^2$ 
centered at $p$, with the radius $\rho$ small so that 
$s(\Omega \cap  B(p,\rho))$ is far from  $\partial \Sigma$.

Let $\wt \Sigma_n $ be the $H$-section obtained by translating $\Sigma$ by the
integral curves of $\xi$ so that $s(p_n)$ goes to
$s_0(p_n)=\wt p_n$. After passing to a subsequence, we can assume the tangent
planes $T_{\wt p_n} \wt \Sigma_n$ converge to a 2-plane $P_{\wt p}$
of the tangent space $T_{\wt p} M^3$; here $\wt p= s_0 (p)$.

We deduce from the Main Theorem, Proposition \ref{P.graphe uniforme} 
($H$-sections are stable) and a
continuity
argument that there exist two real numbers $\delta, \delta_0>0$ such that for
$n$ big enough, a part of $\wt \Sigma_n$  is the  
Euclidean graph (in the sense of Remark \ref{R.Graphe euclidien})  of some
function $\wt u_n$ over the disk of $P_{\wt p}$
centered at $\wt p$ with Euclidean radius $\delta$. Furthermore this part of
$\wt \Sigma_n$ contains the geodesic disk, denoted by $\wt D_n$, centered 
at $\wt p_n$ with radius $\delta_0$.

 Since the mean curvature
of each $\wt D_n$ is constant and equal to $H$, the functions $\wt u_n$ satisfy
an elliptic PDE. Therefore, using the Schauder theory,  we can find a
subsequence of the previous subsequence converging for the $C^k$-topology,
for any $k\in \N$, to some 
$H$-surface $D_{\wt p}$ passing 
through $\wt p$ with tangent plane $P_{\wt p}$, 
$T_{\wt p} D_{\wt p}=P_{\wt p}$. Since  $D_{\wt p}$ is the limit of vertical
graph and $p=\Pi (\wt p)$ is a boundary point of $\Omega$, 
the tangent plane at $\wt p$ must be vertical.
Furthermore $D_{\wt p}$ contains the geodesic
disk centered at $\wt p$ with radius $\delta_0/2$. Denoting by $N$ the unit
normal field along $D_{\wt p}$ given by the limit of the unit normal fields
along the disks $\wt D_n$, observe that the function $g(N,\xi)$ is a Jacobi
function on $D_{\wt p}$. As each $\wt D_n$ is a vertical graph, this function
has a sign, but it vanishes at an interior point of $D_{\wt p}$: $\wt p$. 
Therefore, 
the maximum principle, see Spivak \cite[Chapter 10, Addendum 2, Corollary
19]{[Spivak]}
shows that 
$g(N,\xi)$ is the null function. This means that the normal field $N$ is 
horizontal.  Clearly, this implies that the limit surface $D_{\wt p}$ is a part
of the vertical cylinder over some curve $\wt \alpha \subset M^2$. Since $D_{\wt
p}$
has mean curvature $H$ and the fibers are geodesic lines of $M^3$, the curve
$\wt \alpha$ must be smooth and must have geodesic curvature $2H$ in $M^2$.

Since each $\wt D_n$ is, up to a translation along the fibers, part of the
graph $\Sigma$ over $\Omega$, we deduce that:
\begin{enumerate}
\item $\wt \alpha \subset \ov \Omega \subset M^2$.
\item Each converging subsequence of $\wt D_n$ must converge to the same 
$H$-surface $D_{\wt p}$.
\item The limit surface $D_{\wt p}$ doest not depend on the sequence $(p_n)$
of $\Omega$ converging to $p\in \gamma$.
\end{enumerate}


\medskip 

Now we show that we have $\wt \alpha \subset \gamma \subset \partial \Omega$,
that is, the arc $\gamma$ is smooth and has geodesic curvature $2H$.

Assume by contradiction that there is an interior point 
$q\in \wt \alpha \cap \Omega$. Take any point $\wt q \in D_{\wt p}$ in the
fiber over $q$, $\Pi( \wt q)=q$. By construction, $\wt q$ is the limit of some
sequence $(\wt q_n)$ with $\wt q_n \in \wt D_n$. Since the sequence 
 $(\wt D_n)$ converges to (a part of) the vertical cylinder 
$\Pi^{-1} (\wt \alpha)$ with the $C^k$-topology for any $k\in \N$, the surface
$\Sigma$ must be vertical at $s(q)$, contradicting the fact that 
$\Sigma$ is transversal to the fibers, see Remark
\ref{R.graphes}-(\ref{item.transverse}).

The assertion about the convexity of the arc $\gamma$ is obvious.
\end{proof}

\begin{rem}
In the case where $M^3$ is the metric product space $\hd \times \R$ or 
$\sd \times \R$, the result above was shown in an analytic way by 
 Hauswirth,  Rosenberg and  Spruck \cite{[H-R-S]}.

\end{rem}

\begin{rem}\label{R.Estimes de gradient}
 In a Riemannian product $(M^3,g)=(M^2, h) \times \R$, consider a domain 
$\Omega\subset M^2$ and a smooth surface $\Sigma\subset M^3$ which 
is the vertical graph of a function $u$ on $\Omega$. Let $N$ be a unit normal
field on $\Sigma$ and let $\xi=\frac{\partial}{\partial t}$ be the unit
vertical field. Then we have 
\begin{equation*}
\vert  g(N,\xi) \vert =\frac{1}{\sqrt{1 + \vert \nabla_h u \vert ^2 }}.
\end{equation*}
Therefore, bounding $\vert \nabla_h u \vert$ from above is equivalent to
bounding
$ \vert  g(N,\xi) \vert $ from below away from 0.
\end{rem}

As a second application we obtain {\em interior gradient estimates}, see 
Remark \ref{R.Estimes de gradient}, for $H$-sections.

\begin{thm}
 Let $\Pi : (M^3,g) \rightarrow (M^2,h)$ be a Killing submersion.
Let $\Omega \subset M^2$ be a relatively
compact domain and 
$s_0 : \ov\Omega \rightarrow  \Sigma_0$ a $C^0$-section
over $\ov \Omega$.

 Then, for any $C_1,C_2 >0$, there exists a constant 
$\alpha=\alpha (C_1,C_2,\Omega)$ such that 
for any $p\in \Omega$ with $d(p,\partial \Omega)>C_2$ and 
for any $H$-section
$s : \ov\Omega \rightarrow  \Sigma \subset M^3$ over $\ov \Omega$ 
satisfying $\vert s-s_0 \vert <C_1$ on $ \Omega$, we
have 
\begin{equation}
 \big\vert g\big(N, \xi  \big) (s(p))\big\vert > \alpha,
\end{equation}
where $N$ is a unit normal field along $\Sigma$.

\end{thm}

\begin{proof}
Assume by contradiction that there exist positive constants $C_1$ and $C_2$
such that for any $n\in \Ne$ there exists a $H_n$-section 
 $s_n : \ov\Omega \rightarrow  \Sigma_n \subset M^3$ over 
$\ov \Omega$ with $\vert s_n-s_0 \vert <C_1$, 
and there exists a
point $p_n \in \Omega$ such that $d(p_n ,\partial \Omega)>C_2$  and verifying
\begin{equation}\label{Eq.vertical}
 \big\vert g\big(N_n, \xi  \big) (s_n(p_n))\big\vert \leq \frac{1}{n},
\end{equation}
where $N_n$ is the unit normal field along $\Sigma_n$ oriented by $\vec H_n$.

Since $\ov \Omega$ is compact, there exists a subsequence of $(p_n)$ converging
to a point $p\in \Omega$, with $d(p,\partial \Omega)\geq C_2$. For 
each $n\in \Ne$ we denote by $\wt \Sigma_n$ the vertically translated copy of
$\Sigma_n$ passing through $\wt p_n :=s_0 (p_n)$. Observe that the sequence 
$(\wt p_n)$ converges to $\wt p:=s_0 (p)$.

 Since each $H_n$-surface $\wt \Sigma_n$ is a vertical graph, $\wt \Sigma_n$ is
 stable. Thus we can apply our Main Theorem and
Proposition \ref{P.graphe uniforme}. Therefore, there exist positive constants
$\delta$ and $\delta_0$ such that for each $n\in \Ne$ a part $\wt S_n$ of 
$\wt \Sigma_n$  is a Euclidean graph over the disk of
 $T_{\wt p_n} \wt \Sigma_n$ centered at $\wt p_n$ with Euclidean radius
$\delta_0$, furthermore $\wt S_n$ contains the geodesic disk $\wt D_n$ of 
$\wt \Sigma_n$ centered at $\wt p_n$ with radius $\delta$.

As usual, up to choosing a subsequence, we can assume that the sequence 
of tangent planes $(T_{\wt p_n} \wt \Sigma_n)$  converges to a $2$-plane 
$P_{\wt p}\subset T_{\wt p} M^3$ at $\wt p$.

Taking into account (\ref{Eq.vertical}), we deduce that 
$P_{\wt p}$ is vertical. Thus, there exists a positive number
$\delta_0^\prime < \delta_0$ such that 
for $n$ big enough, a part of 
$\wt S_n$ is a Euclidean graph over the disk of
 $P_{\wt p}$ centered at $\wt p$ with Euclidean radius
$\delta_0^\prime$ and this part contains the  geodesic disk  of 
$\wt \Sigma_n$ centered at $\wt p_n$ with radius $\delta/2$.

Observe that the sequence $( H_n )$ is bounded. Indeed, consider the geodesic
sphere $S(\wt p)$ in $M^3$ centered at $\wt p$ with radius $C_2/2$ and  
$H_0 >0$ an upper bound of the absolute mean curvature of $S(\wt p)$.

Using the comparison principle applied to $\Sigma_n$ and a suitable 
translated copy of $S(\wt p)$ along the fibers, we get 
$\vert H_n \vert \leq H_0$. Up to taking a subsequence, we can assume that 
$H_n \to H$.

As in the proof of the Main Theorem, using the Schauder theory, we can prove
that a subsequence of $(\wt S_n)$ converges in the $C^k$-topology, for 
any $k\in \N$, to a
$H$-surface $\wt S$ passing through $\wt p$ with vertical tangent plane 
$P_{\wt p}=T_{\wt p}\wt S$. As in the proof of Theorem \ref{T.courbe}, the
maximum principle, see Spivak \cite[Chapter 10, Addendum 2, Corollary
19]{[Spivak]}, shows
that $g(N,\xi)$ is the null function on $\wt S$, where $N$ is the limit unit
normal field on $\wt S$. This means that $\wt S$ is part of a
vertical $H$-cylinder over
some curve 
$\wt \gamma \subset \Omega$, that is, $\wt S \subset \Pi^{-1}(\wt \gamma)$.
Since $\wt S$ is a $H$-surface, $\wt \gamma$ is a smooth curve with 
geodesic curvature $2H$. Finally, $\wt S$ contains the geodesic disk of 
$\Pi^{-1}(\wt \gamma)$ centered at $\wt p$ with radius $\delta/2$. 

Let us call $\wt q \in \Pi^{-1}(p) \subset M^3$ the point in the
same fiber 
than $\wt p$ with vertical distance $\delta/4$, lying over 
 $\wt p$. Observe that, by construction, $\wt q$ is the limit of some sequence
$(\wt q_n)$, with $ \wt q_n \in \wt D_n \subset \wt \Sigma_n$. 
 Observe that $\Pi (\wt q)\in \Omega$ and 
$d(\Pi (\wt q),\partial \Omega)\geq C_2$ since $\Pi (\wt q)=p$. Therefore,
the same arguments used above show that the geodesic disk of the
vertical $H$-cylinder 
$\Pi^{-1}(\wt \gamma)$ centered at $\wt q$ with radius $\delta/2$ is the limit 
of a sequence $( \Sigma_n^\prime)$, $\Sigma_n^\prime \subset \wt \Sigma_n$,
extending in this way 
the part $\wt S$ of $\Pi^{-1}(\wt \gamma)$. 

Repeating this argument, we can show that a connected part  of the
vertical $H$-cylinder 
$ \Pi^{-1}(\wt \gamma)$, which is as high as we want,
is contained in the limit
set of the sequence $(\wt \Sigma_n)$. This gives a contradiction since
the vertical distance between $\Sigma_0$ and $\Sigma_n$ is 
uniformly bounded by $C_1$ by hypothesis.
\end{proof}

\section{Appendix}


\begin{prop}\label{P.geometrie}
 Let $U\subset \R^3$ be an open set and let $S\subset U$ be an immersed
$C^2$-surface without boundary. Let $g$ be a metric on $U$, denote by $g_{euc}$
the Euclidean metric. Let  respectively $A$ and $\ov A$ be the second
fundamental forms of $S$ for the metrics $g_{euc}$ and $g$. Assume that 
$\vert \ov A \vert <C$ for some constant $C>0$. 

Then, for any $n\in \Ne$, there exists a constant $C_n>0$ which does not depend
on $S$ such that if $\Vert g_{ij}-g_{euc, ij} \Vert_{C^1(U)} < C_n$, 
$1\leq i, j \leq 3$, then for any $p\in S$ and any nonzero tangent vector
$v\in T_pS$  we have:
\begin{equation}
  \vert \ov\lambda_p (v) - \lambda_p (v)\vert < \frac{1}{n}\quad \text{and}
\quad \vert H- \ov H \vert (p) <\frac{1}{n},
\end{equation}
where $\lambda_p (v)$, resp. $\ov\lambda_p (v)$, denotes the normal
curvature of $S$ at $p$ in the tangent direction $v$ with respect to the metric
$g_{euc}$, resp. $g$, and $H$, resp. $\ov H$, denotes the mean curvature of $S$
at $p$ for the metric
$g_{euc}$, resp. $g$, $($the curvatures of both surfaces being computed with
respect to normals inducing the same transversal orientation$)$.

Consequently, for any $n\in \Ne$, there exists a constant $C_n^\prime>0$ which
does not depend on $S$ such that if 
$\Vert g_{ij}-g_{euc, ij} \Vert_{C^1(U)} <C_n^\prime$, 
$1\leq i, j \leq 3$, then we have:
\begin{equation}\label{Eq.bornee}
\Big\vert  \vert A \vert - \vert \ov A \vert \Big\vert <\frac{1}{n} .
\end{equation}
\end{prop}

\begin{proof}
Let us denote the normal curvatures by $\lambda (v)$ and 
$\ov \lambda (v)$. 

Observe that if we have $\Vert g_{ij}-g_{euc, ij} \Vert_{C^1(U)} <C$ for some
constant $C>0$ then, for any Euclidean change of coordinates we certainly have 
$\Vert g_{ij}-g_{euc, ij} \Vert_{C^1(U)} <9C$ in the new coordinates.

We can choose Euclidean coordinates 
$(x_1,x_2,x_3)$ of $\R^3$ such that the origin coincides with $p$,
the tangent plane $T_p S$ coincides with the plane $\{ x_3=0\}$ and
 $v$ is tangent to the $x_1$-axis. 
Thus, a
neighborhood $S_p$ of $p$ in $S$ is the graph of a function $x_3=u(x_1,x_2)$
defined 
in a neighborhood $V$ of the origin in the plane $\{ x_3=0\}$. 
Therefore a parametrization of $S_p$ is given by 
$(x_1,x_2) \mapsto F(x_1,x_2)=(x_1,x_2,u(x_1,x_2))$, 
$(x_1,x_2) \in V \subset \{ x_3=0\} $.

Let us set $F_i=\frac{\partial F}{\partial x_i}$, 
$u_i=\frac{\partial u}{\partial x_i}$, $i=1,2$. We have
\begin{equation}\label{Eq.tangent1}
 u_i(0)=0,\quad i=1,2.
\end{equation}

We denote by
$\ov N$ the Gauss map of $S$ for the metric $g$ and by $N$ the Gauss map of $S$
for the metric $g_{euc}$, both oriented so that at 0 we have 
$N(0)=(0,0,1)$ and $g_{euc}(N,\ov N)(0)>0$. 

We have by definition: 
\begin{equation*}
 \ov \lambda (v)= \frac{g(\ov N, \ov\nabla _{F_1}F_1)}{g(F_1,F_1)}(0)\quad
\text{and} \quad \lambda (v)= \frac{u_{11}}{1+u_1^2}(0),
\end{equation*}
where $\ov \nabla$ stands for the Levi-Civita connection with respect to the
metric $g$.

Using (\ref{Eq.tangent1}) we get $F_1(0)= \partial_{x_1}$. Furthermore a
straightforward computation shows that: 
\begin{equation*}
 \ov N(0)=\frac{1}{\sqrt{g^{33}}}
\big(g^{13}\partial_{x_1} + g^{23}\partial_{x_2} +
g^{33}\partial_{x_3}\big)(0)\quad 
\text{and}\quad
\big( \ov\nabla_{F_1}
F_1\big)(0)=\big(\ov \nabla_{\partial_{x_1}}\partial_{x_1}\big)
(0) + u_{11}(0)\partial_{x_3},
\end{equation*}
Therefore we deduce that
\begin{align}
 \ov \lambda (v) &=\frac{1}{g_{11}\sqrt{g^{33}}}\Big(
g(\ov \nabla_{\partial_{x_1}}\partial_{x_1}, g^{13}\partial_{x_1} +
g^{23}\partial_{x_2} +
g^{33}\partial_{x_3})(0) + u_{11}(0) \Big),\label{Formule1}\\
 \lambda (v) &= u_{11}(0).\label{Formule2}
\end{align}
Observe that  if the
metrics $g$ and $g_{euc}$ are close enough on $U$ in the $C^1$-topology 
then the expression $ g(\ov \nabla_{\partial_{x_1}}\partial_{x_1},
g^{13}\partial_{x_1} +
g^{23}\partial_{x_2} +
g^{33}\partial_{x_3})(0)$ is as close to 0 as we want
and the expression 
 $\frac{1}{g_{11}\sqrt{g^{33}}}$ 
is as close to 1 as we want.
 Observe that these
expressions are independant of $u$, that is, independant of the surface $S$.
By hypothesis, we have $\vert \ov \lambda (v) \vert <C$.
Therefore there exists a constant $D_n>0$ which does not depend on $S$  
such that $\Vert g_{ij}-g_{euc, ij} \Vert_{C^1(U)} < D_n$, 
$1\leq i, j \leq 3$, implies that $\vert u_{11}(0) \vert <2C$.
Consequently we have:
\begin{equation*}
 \vert \ov \lambda (v)- \lambda (v) \vert  <
2C \Big \vert 1- \frac{1}{g_{11}\sqrt{g^{33}}}\Big\vert
+ \frac{1}{g_{11}\sqrt{g^{33}}}
 \Big\vert g(\ov \nabla_{\partial_{x_1}}\partial_{x_1},
g^{13}\partial_{x_1} +
g^{23}\partial_{x_2} +
g^{33}\partial_{x_3})(0) \Big\vert.
\end{equation*}
For the same reasons as before, there exists a constant $C_n>0$, with
$C_n<D_n$, which does not depend on $S$ such that
$\Vert g_{ij}-g_{euc, ij} \Vert_{C^1(U)} < C_n$, 
$1\leq i, j \leq 3$, implies that
\begin{equation*}
  \vert \ov \lambda (v)- \lambda (v) \vert  < \frac{1}{n}.
\end{equation*}
Let us denote by $\ov \lambda_{\text{max}}, \ov \lambda_{\text{min}},
\lambda_{\text{max}}$ and $\lambda_{\text{min}}$ the principal curvatures of $S$
at $P$ for the metrics $g$ and $g_{euc}$ respectively. It follows that:
\begin{equation*}
 \vert \ov \lambda_{\text{max}}- \lambda_{\text{max}} \vert  < \frac{1}{n}
\qquad  \text{and} \qquad
 \vert \ov \lambda_{\text{min}}- \lambda_{\text{min}} \vert  < \frac{1}{n}.
\end{equation*}
Therefore:
\begin{equation*}
 \vert H- \ov H \vert <\frac{1}{n}.
\end{equation*}

Let $k>0$ be some integer and let $C_k$ be the  constant given in the first
part of the proposition. Thus, if  
$\Vert g_{ij}-g_{euc, ij} \Vert_{C^1(U)} < C_k$, 
$1\leq i, j \leq 3$, then we have:
\begin{equation*}
  \vert \ov \lambda_{\text{max}}- \lambda_{\text{max}} \vert  < \frac{1}{k}
\quad \text{and} \quad
 \vert \ov \lambda_{\text{min}}- \lambda_{\text{min}} \vert  < \frac{1}{k}
\end{equation*}
on $S$. Therefore we get
\begin{align*}
 \Big\vert  \vert A \vert^2 - \vert \ov A \vert ^2\Big\vert &=
 \Big\vert \big( \lambda_{\text{max}}^2 +  \lambda_{\text{min}}^2\big) - 
\big( \ov  \lambda_{\text{max}}^2 +  \ov\lambda_{\text{min}}^2\big)
\Big\vert \\
 & \leq \frac{1}{k}\big( 2 \vert \ov\lambda_{\text{max}}\vert +\frac{1}{k}
 +2 \vert \ov\lambda_{\text{min}}\vert +\frac{1}{k}\big) \\
 & < \frac{1}{k}\big( 4C +\frac{2}{k}\big).
\end{align*}
To conclude the proof just observe that 
$\Big\vert \vert A \vert - \vert \ov A\vert \Big\vert ^2\leq 
 \Big\vert  \vert A \vert^2 - \vert \ov A \vert ^2\Big\vert$.
\end{proof}

\begin{rem}\label{R.Graphe euclidien}
 Let $(M,g)$ be a smooth Riemannian 3-manifold and let 
$S \looparrowright M$ be an immersed surface. Let $p\in S$ be an interior point
of $S$. Let $(U,\varphi,B(p,r_0))$ be a harmonic coordinate chart at $p$, where 
$B(p,r_0)$ is the geodesic ball in $M$ centered at $p$ with radius $r_0$ such
that $B(p,r_0) \cap \partial S =\emptyset$. Let us denote the harmonic
coordinates on U by $(x_1,x_2,x_3)$ and let us consider the Euclidean
 metric $dx_1^2 + dx_2^2 + dx_3^2$  on $U$. Viewing the connected component of
$B(p,r_0) \cap  S$ through $p$ in $U$ observe that a neighborhood of 
$p$ in $S$ is a Euclidean graph over a Euclidean disk in the tangent plane
$T_pS$ centered at $p$.
\end{rem}

\bigskip

The arguments of the proof of the Proposition \ref{P.geometrie} can be used to
show the following result which is very useful for the applications.

\begin{prop}\label{P.graphe uniforme}
Let $(M,g)$ be a smooth Riemannian 3-manifold 
$($not necessarily complete$)$, 
with
bounded sectional curvature 
$\vert K \vert \leq \Lambda < +\infty$ and let  $r>0$. Let
$\Omega\subset M$ be an open subset of $M$ such that  
the injectivity radius in $M$ at any $x\in \Omega$ is $\geq r $. 

Then  for any $C_1>0$ and  $C_2>0$ there exist constants 
$\delta, \delta_0 >0$ depending only on 
$C_1,C_2,\Lambda$ and $r$ and neither on $M$ nor on $\Omega$ satisfying the
following:

For any immersed surface $S \looparrowright \Omega$ whose 
second fundamental form $\ov A$ satisfies $\vert \ov A\vert <C_1$  and for
any $p\in S$ such that $d_S (p,\partial S)>C_2$ then 
a part $S_0$ of $S$ is contained in the image of a harmonic chart and 
is a Euclidean graph $($see Remark $\ref{R.Graphe euclidien})$
over the disk of 
$T_pS$ centered at $p$ with Euclidean radius $\delta$. Furthermore,
the subset $S_0$ contains the
geodesic disk of $S$ centered at $p$ with radius $\delta_0$.
\end{prop}

\begin{proof}
Let us fix some constants $\alpha \in (0,1)$ and $Q_0 >1$.
The Theorem 6 in Hebey-Herzlich \cite{[Heb-Her]} shows that there exists some 
$r_0=r_0(Q_0,\alpha,r, \Lambda   ) >0$ such
that 
for any $x\in \Omega$ there exists a harmonic coordinate chart
$(U,\varphi,B(x,r_0))$ satisfying the assumption of the Theorem
\ref{T.cartes harmoniques} where
$B(x,r_0)$ is the geodesic ball in $M$ centered at $x$ with radius $r_0$.

 Let $C_1,C_2$, $S$ and   $p\in S$ be as stated. Consider a harmonic 
coordinate chart
$(U,\varphi,B(p,r_0))$ at $p$, and call $x=(x_1,x_2,x_3)$ the coordinates on 
$U\subset \R^3$, we have 
\begin{equation*}
 \sum_{k=1}^3\sup _{x\in U} \vert \partial_k g_{ij}(x) \vert \leq Q_0 
\quad \text{and} \quad   Q_0^{-1} \delta_{ij} \leq g_{ij} \leq Q_0 \delta_{ij}
\quad \text{as
quadratic forms}.
\end{equation*}
We want to show that there exists a constant 
$C_3=C_3(C_1)$ which does not depend on $p\in S$ or on $S$
such that $\vert A (p) \vert <C_3$ where $A$ denotes the second fundamental
form 
of $S \cap B(p,r_0)$ for the Euclidean metric $dx_1^2 + dx_2^2 + dx_3^2$.
 Let $v \in T_pS$ be a nonzero tangent vector. Choose Euclidean coordinates
$y=(y_1,y_2,y_3)$ in $U$ such that the tangent plane $T_pS$ coincides
with the plane $\{y_3=0\}$ and $v$ is tangent to the $y_1$-axis.
In those new coordinates we have 
\begin{equation}\label{Formules}
 \sum_{k=1}^3\sup _{y\in U} \vert \partial_k g_{ij}(y) \vert \leq 9 Q_0 
\quad \text{and} \quad   Q_0^{-1} \delta_{ij} \leq g_{ij} \leq Q_0 \delta_{ij}
\quad \text{as
quadratic forms}.
\end{equation}
where $\partial_k$ stands for the partial derivative with respect to $y_k$.

Let us denote by $\lambda_p (v)$, resp. $\ov\lambda_p (v)$, the normal
curvature of $S$ at $p$ in the tangent direction $v$ with respect to the
Euclidean metric
$g_{euc}$, resp. $g$. We 
get from the formulae (\ref{Formule1}) and (\ref{Formule2}) established in the
proof of the Proposition \ref{P.geometrie} that 
\begin{equation*}
\lambda_p (v) =
(g_{11}\sqrt{g^{33}})(0) \ov  \lambda (v)- 
g(\ov \nabla_{\partial_{y_1}}\partial_{y_1}, g^{13}\partial_{y_1} +
g^{23}\partial_{y_2} +
g^{33}\partial_{y_3})(0).
\end{equation*}
Taking into account the formulae (\ref{Formules}) we get that the
expressions
$(g_{11}\sqrt{g^{33}})(0)$ and 
$\vert g(\ov \nabla_{\partial_{y_1}}\partial_{y_1}, g^{13}\partial_{y_1} +
g^{23}\partial_{y_2} +
g^{33}\partial_{y_3})(0) \vert$ are bounded
by a constant $Q_1$ which only depends on $Q_0$. Therefore the 
expression $\vert \lambda_p (v) \vert $ is bounded 
by a constant $C_4$ which depends only on $C_1$ and $Q_0$ and which does not
depend on $p\in S$  or on $S$. Consequently we have 
$\vert A\vert <2C_4$ on $S$.
 Now the proof
follows from the Proposition \ref{P.graph} setting 
$\delta =1/8C_4$. The second part of the estimates (\ref{Formules}) shows that
we
can choose $\delta_0=\delta/Q_0$.

\end{proof}


\begin{thebibliography}{10000}


\bibitem{[Be-Haus]} P. B\'erard \& L. Hauswirth. {\em General curvature
estimates for stable $H$-surfaces immersed into a space form},  J. Math. Pures
Appl. 78,  no. 7, 667--700, 1999.

\bibitem{[C-M]} T.H. Colding \& W.P. Minicozzi. {\em Minimal Surfaces},
Courant Lecture Notes 4, 1999. 
 

\bibitem{[Cold-Min2]} T.H. Colding \& W.P. Minicozzi. {\em Estimates
for parametric elliptic integrands},  Int. Math. Res. Not.  no. 6,
291--297, 2002.

\bibitem{[DoC1]} M. do Carmo. {\em Riemannian geometry}. Translated from the
second Portuguese edition by Francis Flaherty. Mathematics: Theory \&
Applications. Birkhäuser Boston, Inc., Boston, MA, 1992.



\bibitem{[DoC-Pe]} M. do Carmo \& C.K. Peng. {\em Stable complete minimal
surfaces in $R\sp{3}$ are planes},  Bull. Amer. Math. Soc. (N.S.)  1,
no. 6, 903--906, 1979. 

\bibitem{[FC-Sc]} D. Fischer-Colbrie \& R. Schoen. {\em The structure of
complete stable minimal surfaces in $3$-manifolds of nonnegative scalar
curvature},  Comm. Pure Appl. Math.  33, no. 2, 199--211, 1980.


\bibitem{[Gi-Tr]} D. Gilbarg \& N.S. Trudinger. {\em Elliptic partial
differential equations of second order}, Reprint of the 1998 edition. Classics
in Mathematics. Springer-Verlag, Berlin, 2001. 




\bibitem{[H-R-S]} L. Hauswirth \& H. Rosenberg \& J. Spruck. 
{\em Infinite boundary value problems for constant mean curvature graphs in 
$\hd \times \R$ and $\sd \times \R$}, to appear in Amer. Jour. Math.


\bibitem{[Heb-Her]} E. Hebey \& M. Herzlich. {\em Harmonic coordinates, harmonic
radius and convergence of Riemannian manifolds},  Rend. Mat. Appl. (7)  17,
no. 4, 569--605, 1997.

\bibitem{[Heinz]} E. Heinz. {\em Über die Lösungen der
Minimalflächengleichung}, Nachr. Akad. Wiss. Göttingen. Math.-Phys.
Kl. Math.-Phys.-Chem. Abt., 51--56, 1952.

\bibitem{[Lopez-Ros]} F. López \& A. Ros {\em Complete minimal
surfaces with index one and stable constant mean curvature surfaces},  Comment.
Math. Helv.  64,  no. 1, 34--43, 1989.


\bibitem{[MPR]} W. Meeks \& J. P\'erez \& A. Ros. {\em Stable constant mean
curvature surfaces}, Advanced Lectures in Mathematics, Vol. 7,
Handbook of Geometric Analysis, No. 1, International Press, 301--380,
2008. 

\bibitem{[Nel-Ros1]} B. Nelli \& H. Rosenberg. {\em Minimal
surfaces in ${\mathbb H}\sp 2\times\mathbb R$},  Bull. Braz. Math. Soc. (N.S.) 
33, 
no. 2, 263--292, 2002.



\bibitem{[Nel-Ros2]} B. Nelli \& H. Rosenberg. {\em Global
properties of constant mean curvature surfaces in 
$\mathbb H\sp 2\times\mathbb R$}, 
Pacific J. Math. 226,  no. 1, 137--152, 2006.


\bibitem{[P-R]} J. P\'erez \& A. Ros. {\em Properly embedded minimal surfaces 
with finite total curvature},  The global theory of minimal surfaces in flat
spaces (Martina Franca, 1999), Lecture Notes in Maths.1775, 15--66,
Springer, Berlin, 2002. 


 \bibitem{[Pet]} P. Petersen. {\em Riemannian geometry}, Graduate Texts in
Mathematics 171, New York Springer, 2006.

\bibitem{[Pog]} A. Pogorelov. {\em On the stability of minimal surfaces},
Soviet Math. Dokl., 24, 274--276, 1981.

\bibitem{[Ros]} A. Ros. {\em One-sided complete stable minimal surfaces},
J. Differential Geom. 74, no. 1, 69--92, 2006.

\bibitem{[Sch]} R. Schoen. {\em Estimates for stable minimal surfaces in
three-dimensional manifolds},  Seminar on minimal submanifolds,  111--126, Ann.
of Math. Stud., 103, Princeton Univ. Press, Princeton, NJ, 1983. 

\bibitem{[Sil]} A.M. Da Silveira. {\em Stability of complete noncompact
surfaces with constant mean curvature},  Math. Ann.  277,  no. 4,
629--638, 1987.

\bibitem{[Spivak]} M. Spivak. {\em A comprehensive introduction to
    differential geometry}, Vol 5, Boston, Publish or Perish, 1979.



\bibitem{[Steenrod]} N. Steenrod. {\em The Topology of Fibre Bundles},
Princeton Mathematical Series, vol. 14. Princeton University Press, Princeton,
N. J., 1951



\bibitem{[Zh]} S. Zhang.
{\em Curvature estimates for CMC surfaces in three dimensional manifolds},
Math. Z. 249, no. 3, 613--624, 2005.

\end{thebibliography}
\end{document}